\documentclass[12pt]{amsart}
\usepackage[utf8]{inputenc}
\usepackage[T1]{fontenc}
\usepackage{amsfonts}
\usepackage{amsmath,amscd}
\usepackage{amsthm}
\usepackage{latexsym}
\usepackage{amssymb}
\usepackage{enumerate}
\usepackage{url}
\markleft{\hfill G. Raposo\hfill }
\usepackage{hyperref}

\newtheorem{theorem}{Theorem}[section]
\newtheorem*{theorem*}{Theorem}
\newtheorem{proposition}[theorem]{Proposition}

\newtheorem{definition}[theorem]{Definition}

\newtheorem{lemma}[theorem]{Lemma}
\newtheorem{remark}[theorem]{Remark}
\newtheorem*{claim*}{Claim}
\usepackage{tikz-cd}

\usepackage{mathdots}

\usepackage{geometry}
\usepackage{tikz}
\usetikzlibrary{arrows,matrix,positioning}

\usepackage{caption}
\usepackage{subcaption}
\usepackage{graphicx}

\numberwithin{equation}{section}

\newcommand{\Z}{\mathbb{Z}} 

\allowdisplaybreaks

\title{Circular Rearrangement Inequality and Optimal Cyclic Birth and Death Chains}
\author{Gabriel Raposo}
\address{Department of Mathematics, University of Toronto, \newline 40 St. George Street, Toronto, ON M5S2E4, Canada}
\email{gabo.raposo@utoronto.ca}

\begin{document}

\maketitle

\begin{abstract}
We prove a generalized circular rearrangement inequality: among all circular arrangements of a finite collection of positive numbers, the greedy arrangement simultaneously maximizes the sums of products of $k$ consecutive entries for every $k$. This resolves a 2023 conjecture of Holmes--Holroyd--Ramírez and extends the classical circular rearrangement inequality for products of adjacent entries. As an application, we show that the greedy ordering minimizes the speed of a cyclic birth-and-death chain, resolving another conjecture of the same authors. \end{abstract}


\section{Introduction}

We consider a random walk on $\Z$ with transition probabilities $\big(p_{i,i+1}\big)_{i\in \Z}\in (0,1)^\Z$ satisfying $p_{i,i+1}=p_{i\,\textup{mod } n}$ for some integer $n\geq 1$ and $\vec{p}=(p_1,\dots,p_{n})$. These models of \textit{random walks in periodic environments} have been studied in higher dimensions \cite{Tak02,KU08,RV10}. In the one dimensional setting we refer to such random walks as \textit{cyclic birth and death chains}. In this setting a simple criterion for transience is given by means of the quantity
\begin{equation*}
    \gamma:=\prod_{i=1}^{n} \frac{1-p_i}{p_i}.
\end{equation*}
\begin{proposition}[\cite{HHR23}]
Let $X$ be a cyclic birth and death chain, then we define the speed of $X$ to be the limit
\begin{equation*}
    s(\vec{p}):=\lim_{m\to \infty} \frac{X_m}{m}.
\end{equation*}
It exists almost surely and is deterministic. Moreover, $s(\vec{p})>0$ if and only if $\gamma<1$, $s(\vec{p})<0$ if and only if $\gamma>1$, while $s(\vec{p})=0$ and $X$ is recurrent if and only if $\gamma=1$. 
\end{proposition}

We are interested in studying the speed of such Markov chains, for simplicity we restrict ourselves to the case $\gamma<1$ when $s(\vec{p})>0$. In this paper we describe the optimal rearrangement of the transition probabilities $p_1,\dots,p_n$ which minimizes the speed. This question was originally posed in \cite[Conjecture 8]{HHR23}, where an explicit formula for the speed is obtained and the number of distinct possible speeds induced by different rearrangements is described. Related optimization questions for random walks in fixed environments
have been studied in special settings; see \cite{PR12,CMT20}. A crucial object of study are the \textit{circular rearrangement products}. 

\begin{definition}
Let $1\leq k \leq n$ integers, $x_1, x_2, \dots , x_n$ positive real numbers and $\sigma \in S_n$ a permutation, define the $(\sigma,k)$-product to be 
\begin{equation*}
P_k(\sigma,\vec{x})= \sum_{i=1}^n \prod_{j=1}^k x_{\sigma(i+j)}.
\end{equation*}
Here the domain of $\sigma$ is extended on the integers periodically. 
\end{definition}

Note that $P_k(\sigma,\vec{x})$ is invariant under the action of the dihedral group. This fact, jointly with an explicit formula for the speed, allows to compute the number of distinct speeds. Indeed, we have the following.  

\begin{theorem}[\cite{HHR23}]\label{TheoremHHR}
Let $\gamma<1$ and $\rho:=\big(\frac{1-p_1}{p_1},\frac{1-p_2}{p_2},\dots,\frac{1-p_n}{p_n}\big)$, then
\begin{equation}\label{eqSpeedy}
    s(\vec{p})=\frac{1-\gamma}{1-\gamma+\frac{2}{n}\sum_{k=1}^n P_k(I_n,\rho)}.
\end{equation}
Here $I_n$ denotes the identity permutation of the symmetric group $S_n$.
\end{theorem}

An immediate consequence of the invariance under the action of the dihedral group of $P_k$ is that $N(\vec{p})\leq \frac{n!}{2n}$ when $n\geq 3$. In fact, in \cite{HHR23} it is proved that, for Lebesgue almost every $\vec{p}\in (0,1)^n$ the number of distinct speeds $N(\vec{p})$ satisfies $N(\vec{p})=\tfrac{n!}{2n}$.

Note that the expression above for the speed is not unique, indeed we also have that $s(\vec{p})=\sum_{i=1}^n (2p_i-1)\pi_i$ where $(\pi_1,\dots,\pi_n)$ denotes the stationary distribution of the Markov chain $X_n^\circ = X_n\,\textup{ mod }n$. While it is expected for the speed to be invariant under rotations it is not immediate that it also is invariant under reflection. Using the formula above one can hope to find an ordering that minimizes the speed. To this end, we introduce the greedy permutation. 

\begin{definition}
Let $n\geq 1$, we define the greedy permutation $\sigma_g\in S_n$ to be 
\begin{equation*}
    \sigma_g=(1,3,5,\dots,2 \lfloor\tfrac{n+1}{2}\rfloor-1,2 \lfloor\tfrac{n}{2}\rfloor,\dots,6,4,2).
\end{equation*}
\end{definition}
For example, for $n=7$ we have $\sigma_g=(1,3,5,7,6,4,2)$. This kind of ordering has been called \textit{circular symmetric ordering} in \cite{Abr67} or \textit{Pendulum Arrangement} in \cite{CMT20}. In fact, the $(\sigma,k)$-product has been previously studied for $k=2$ in \cite{Abr67,Yu18} where it is proven that for $x_1<x_2<\dots<x_n$, $P_2(\sigma,\vec{x})\leq P_2(\sigma_g,\vec{x})$ and the functional extension of the inequality in \cite{AP20}. This is generalized up to $k\leq 3$ and $n-3\leq k \leq n$ in \cite{HHR23}, here we prove this statement for general $k$.

\begin{theorem}[Generalized circular rearrangement inequality]\label{TheoremGenCircRearIneq}
Let $n\geq 1$ and $1\leq k \leq n$ be integers. Let $ x_1\leq x_2\leq \dots \leq x_n$ be positive real numbers. For any permutation $\sigma \in S_n$ and $\sigma_g$ the greedy permutation of $S_n$, we have that
\begin{equation}\label{GenCircRearIneq}
    P_k(\sigma,\vec{x}) \leq P_k(\sigma_g,\vec{x}).
\end{equation}
\end{theorem}

The statement above settles conjecture $5$ in \cite{HHR23}. Additionally, equation (\ref{eqSpeedy}) implies that a permutation maximizing the circular rearrangement products would minimize the speed. Hence conjecture $8$ in \cite{HHR23} follows immediately from Theorem \ref{TheoremGenCircRearIneq}. 

\begin{theorem}[Greedy is least speedy]
For any $\vec{p}$ with non-increasing entries, such
that $\gamma < 1$ (so all possible speeds will be positive), and every $\sigma\in S_n$ we have that
\begin{equation*}
    s\big(\vec{p}_{\sigma_g}\big)\leq s\big(\vec{p}_{\sigma}\big).
\end{equation*}
Here $\vec{p}_\sigma = (p_{\sigma(1)},p_{\sigma(2)},p_{\sigma(3)},\dots,p_{\sigma(n-1)},p_{\sigma(n)})$.
\end{theorem}

Note that one might interpret this statement as saying that the speed is minimized by having a "smooth" ordering, that is, a cyclic arrangement of the elements of $\vec{p}$ that has no large jumps.


The proof of Theorem \ref{TheoremGenCircRearIneq} combines an induction argument based on deleting the largest entry, a symmetry between products of lengths $k$ and $n-k$, and a probabilistic interpretation of an algebraic expression. 

The first ingredient follows \cite{Abr67,Yu18, HHR23} where for fixed $k$ an inductive argument is done in $n$. By decomposing $P_k(\sigma,\vec{x})=P_k(\sigma',\vec{x}')+D_k$ we study the quantity $D_k$. Note that the proof in \cite{Yu18} contains an algebraic error. The second ingredient consists in exploiting the symmetric structure of $P_k$ to replace the induction argument on $n$ with a structural induction argument on both $n$ and $k$. The third ingredient consists in giving a probabilistic interpretation to the quantity $D_k$, this allows us to apply a result of \cite{JHCH05}.

\begin{remark}
Similarly to the problem of minimizing the speed, we would like to find a universal permutation $\sigma \in S_n$ determined only by the ordering of $\vec{p}$ that maximizes the speed. This is unfortunately not possible. First, the circular rearrangement products $P_k$ do not share a common permutation that minimize them simultaneously. One may still hope that rather than minimizing $P_k$ to find the maximal speed, one could attempt to find a permutation minimizing $\sum_{k=1}^n P_k(\sigma,\rho)$. Indeed, starting from $\vec{p}=\left(\frac{4}{5},\frac{2}{3},\frac{4}{7},\frac{1}{2},\frac{4}{9},\frac{2}{5}\right)$ the maximal speed is given for $\sigma_1=(1,5,3,4,2,6)$, for which $s(\vec{p}_{\sigma_1})=\frac{633}{4894}\approx0.129$. While for $\vec{p}'=\left(\frac{6}{7},\frac{3}{4},\frac{2}{3},\frac{3}{5},\frac{3}{8},\frac{3}{53}\right)$, the maximal speed is given for $\sigma_2=(1,4,5,3,2,6)$, for which $s(\vec{p}_{\sigma_2}')=\frac{708}{27953}\approx 0.025$. A direct examination of the $60$ arrangements up to dihedral symmetry shows that the two optimizers are unique. Hence the permutation maximizing the speed cannot be only determined by the order of the transition probabilities.
\end{remark}

\section{Proof of Theorem \ref{TheoremGenCircRearIneq}}

Denote by $\mathcal{D}$ the set of pairs $(n,k)$ with $n\geq 1$ and $1\leq k \leq n$ such that the inequality \ref{GenCircRearIneq} is satisfied for all $\sigma\in S_n$ and $0<x_1<x_2<\dots<x_n$. It is immediate from the definition that $(n,1)\in \mathcal{D}$ and that $(n,n)\in \mathcal{D}$ for all $n\geq 1$.

\begin{lemma}\label{OldLemma}
If $(n,k)\in \mathcal{D}$ and $1\leq k \leq n-1$, then $(n,n-k)\in \mathcal{D}$. 
\end{lemma}

The lemma above is present in \cite[Section 5]{HHR23}. It follows from noticing that $P_{n-k}(\sigma,\vec{x})=P_k(\sigma,x_1^{-1},\dots,x_n^{-1})\prod_{i=1}^n x_i$, which induces a symmetry between $P_k$ and $P_{n-k}$. One of our contributions is the lemma below.

\begin{lemma}\label{NewLemma}
If $(n,k)\in \mathcal{D}$, $k\geq 2$ and $n\geq 2k-2$, then $(n+1,k)\in \mathcal{D}$. 
\end{lemma}

We postpone the proof of the lemma to provide a short proof of Theorem \ref{TheoremGenCircRearIneq}. Note that it is enough to show that $(n,k) \in \mathcal{D}$ for any $1\leq k\leq n$ and $n\geq 1$ to conclude. 

\begin{proof}[Proof of Theorem \ref{TheoremGenCircRearIneq}]
We first show by induction on $k \ge 1$ that $(n,k) \in \mathcal{D}$ for all $n \ge 2k-1$. The base case $k=1$ holds directly from the given conditions. Assuming the claim holds for fixed $k$, we have $(2k+1, k) \in \mathcal{D}$, hence Lemma \ref{OldLemma} then gives $(2k+1, k+1) \in \mathcal{D}$. An iterated application of  Lemma \ref{NewLemma} gives $(n, k+1) \in \mathcal{D}$ for all $n \ge 2(k+1)-1$, completing the main induction. Now consider any pair $(n,k)$ with $1 \le k \le n$,
\begin{itemize}
    \item if $k \le n/2$, then $n > 2k-1$, and we have $(n,k) \in \mathcal{D}$.
    \item If $n>k > n/2$, then $(n,n-k) \in \mathcal{D}$ by the item above, and Lemma \ref{OldLemma} gives $(n, k) \in \mathcal{D}$.
    \item If $k=n$ then $(n,k)\in\mathcal{D}$ as a base case. \qedhere
\end{itemize}\end{proof}

\begin{definition}
Let $k\geq 2$, $0<a_1,\dots,a_{2k-2}<x$ real numbers and $\tau \in S_{2k-2}$. We define the $(\tau,k)$-difference to be 
\begin{equation*}
D_k(\tau,x,\vec{a}):=\sum_{i=1}^{k} x\prod_{j=0}^{k-2} a_{\tau(i+j)}-\sum_{i=1}^{k-1} \prod_{j=0}^{k-1}a_{\tau(i+j)}.
\end{equation*}
\end{definition}

\begin{lemma}\label{LemmaDincreasing}
For fixed $k$, $x$, $\vec{a}=(a_1,a_2,\dots,a_{2k-2})$ and $\tau\in S_{2k-2}$, we have that $D_k(\tau,x,\vec{a})$ is increasing in each coordinate $a_\ell$ for $1\leq \ell \leq 2k-2$.
\end{lemma}
\begin{proof}
Start by writing $b_\ell=a_{\tau(\ell)}$ for $1\leq \ell \leq 2k-2$, and note that for $\vec{b}'=(b_{2k-2},b_{2k-3},\dots,b_2,b_1)$ we have that $D_k(\tau,x,\vec{a})=D_k(I_{2k-2},x,\vec{b})=D_k(I_{2k-2},x,\vec{b}')$. Hence it is enough to verify that $\partial_{b_\ell} D_k(I_{2k-2},x,\vec{b})\geq 0$ for $1\leq \ell \leq k-1$. Indeed, we have that
\begin{equation*}
    \partial_{b_\ell} D_k(I_{2k-2},x,\vec{b})=\sum_{i=1}^\ell (x-b_{k-1+i})\prod_{\substack{j=i\\j\neq \ell}}^{i+k-2} b_{j}>0. \qedhere
\end{equation*}
\end{proof}

\begin{lemma}\label{LemmaDisGreedy}
Let $k\geq 2$, $0<a_1<\dots<a_{2k-2}<x$ real numbers and $\tau \in S_{2k-2}$. Then
\begin{equation*}
    D_k(\tau,x,\vec{a}) \leq D_k(\tau_g,x,\vec{a}).
\end{equation*}
\end{lemma}
\begin{proof}
Start by writing $q_i=a_i/x$ for $i=1,\dots,2k-2$. We then have that
\begin{equation*}
D_k(\tau,x,\vec{a})=x^k \bigg(\sum_{i=1}^{k}\prod_{j=0}^{k-2} q_{\tau(i+j)}-\sum_{i=1}^{k-1} \prod_{j=0}^{k-1}q_{\tau(i+j)}\bigg)=x^kD_k(\tau,1,\vec{q}).
\end{equation*}
Hence, it is enough to maximize $D_k(\tau,1,\vec{q})$. Interestingly, the quantity $D_k(\tau,1,\vec{q})$ has a probability interpretation. Consider a system with $2k-2$ independent components in the order $\big(1,2,3,\dots,2k-3,2k-2\big)$ where $q_{\tau(i)}$ is the probability that the component occupying position $i$ functions. We say that the system functions if at least $k-1$ consecutive components function, then 
\begin{equation*}
D_k(\tau,1,\vec{q})=\prod_{j=1}^{k-1}q_{\tau(j)}+\sum_{i=1}^{k-1}(1-q_{\tau(i)})\prod_{j=1}^{k-1}q_{\tau(i+j)}
\end{equation*}
is precisely the probability that the system functions, also called \textit{reliability} of the system. This kind of systems has been previously studied, it is shown in \cite{JHCH05} that $D_k(\tau,1,\vec{q})\leq D_k(\tau_g,1,\vec{q})$, which immediately provides the conclusion.
\end{proof}

\begin{remark}
Note that the proof that $D_k(\tau,1,\vec{q})\leq D_k(\tau_g,1,\vec{q})$ is elementary, by fixing a transposition $\theta_{i,j}=(i\,j)$, the quantity $D_k(\tau,1,\vec{q})-D_k(\tau\circ\theta_{i,j},1,\vec{q})$ can be simplified. By studying different cases, the proof in \cite{JHCH05} concludes by means of an inductive argument. In fact, rearrangement problems for runs of independent components and their
applications to network reliability go back at least to \cite{Ton85}.
\end{remark}

We can now finish the proof of Lemma \ref{NewLemma} by means of lemmas \ref{LemmaDincreasing} and \ref{LemmaDisGreedy}.

\begin{proof}[Proof of Lemma \ref{NewLemma}]
Fix $k\geq 2$ and $n\geq 1$, $\sigma\in S_{n+1}$ and $\vec{x}\in (0,\infty)^{n+1}$, without loss of generality we can assume that $x_1<x_2<\dots<x_{n+1}$. By using the invariance under rotation of $P_k$ we can assume that $\sigma(n+1)=n+1$. Denote $x:=x_{n+1}$, $c_{k-j}=x_{\sigma(n+1-j)}$ and $c_{k-1+j}=x_{\sigma(j)}$ for $j=1,\dots,k-1$. Note that $\vec{c}$ denotes the $2k-2$ neighbors of $x_{n+1}$ in $P_k(\sigma,\vec{x})$. Let $\vec{a}=(a_1,a_2,\dots,a_{2k-2})$ be the increasing reordering of $\vec{c}$ and let $\tau\in S_{2k-2}$ be the permutation that satisfies $\vec{a}_\tau=(a_{\tau(1)},a_{\tau(2)},\dots,a_{\tau(2k-2)})=\vec{c}$, then we can write
\begin{equation}\label{eq1}
    P_k(\sigma,\vec{x})=P_k(\sigma',\vec{x}')+D_k(\tau,x_{n+1},\vec{a}),
\end{equation}
where $\vec{x}'$ denotes $\vec{x}$ without the $n+1$ coordinate and $\sigma'$ is the permutation $\sigma$ restricted to $S_{n}$. Equivalently, by considering a greedy permutation $\sigma_g$ satisfying $\sigma_g(n+1)=n+1$, then we can write
\begin{equation}\label{eq2}
P_k(\sigma_g,\vec{x})=P_k(\sigma'_g,\vec{x}')+D_k(\tau_g,x_{n+1},\vec{y}).
\end{equation}
Here $\vec{y}$ denotes the vector of the $2k-2$ largest elements of $\vec{x}'$ in increasing order. Note that deleting $n+1$ in the greedy ordering $\sigma_g$ of $S_{n+1}$ still produces a greedy ordering $\sigma_g'$ of $S_n$. Similarly the permutation $\tau_g$ is also a greedy ordering on $S_{2k-2}$.

At this point we are left to compare equations (\ref{eq1}) and (\ref{eq2}). First, the induction hypothesis ensures that $P_k(\sigma',\vec{x}')\leq P_k(\sigma'_g,\vec{x}')$. Note that because $\vec{y}$ is the ordered set of variables next to $x_{n+1}$ in the greedy ordering, we automatically must have that $a_\ell\leq y_\ell$ for $1\leq \ell\leq 2k-2$. Additionally, the hypothesis that $n\geq 2k-2$ ensures that the $2k-2$ neighboring positions around $x_{n+1}$ are distinct, hence Lemma \ref{LemmaDincreasing} gives that $D_k(\tau_g,x_{n+1},\vec{a})\leq D_k(\tau_g,x_{n+1},\vec{y})$. Finally, Lemma \ref{LemmaDisGreedy} gives that $D_k(\tau,x_{n+1},\vec{a}) \leq D_k(\tau_g,x_{n+1},\vec{a})$. We conclude the proof of the statement for all $x_1\leq x_2\leq\dots \leq x_n\leq x_{n+1}$ by continuity of the function $P_k$. 
\end{proof}

\end{document}